\documentclass[12pt, a4paper]{article}
\usepackage{amsmath, amsfonts, amssymb, amsthm}
\usepackage{epsfig}
\usepackage{color}

\usepackage{hyperref}

\newlength{\hchng}
\newlength{\vchng}
\newtheorem{thm}{Theorem}[section]
\newtheorem{prop}[thm]{Proposition}

\newtheorem{lemma}[thm]{Lemma}

\newtheorem{preremark}[thm]{Remark}
\newenvironment{remark}{\begin{preremark}\rm}{\medskip \end{preremark}}
\numberwithin{equation}{section}

\newcommand{\begeqa}{\begin{eqnarray}}
\newcommand{\eneqa}{\end{eqnarray}}
\newcommand{\begeqaet}{\begin{eqnarray*}}
\newcommand{\eneqaet}{\end{eqnarray*}}
\newcommand{\beeq}{\begin{equation}}
\newcommand{\eeq}{\end{equation}}
\newcommand{\beeqs}{\begin{equation*}}
\newcommand{\eeqs}{\end{equation*}}

\newcommand{\rn}{\mathbb{R}^n}

\title{Sharp boundary and global regularity for degenerate fully nonlinear elliptic equations}
\author{ Dami\~{a}o J. Ara\'ujo \; {and} \; Boyan Sirakov }
\date{}

\begin{document}
\maketitle

\begin{abstract}

We obtain optimal boundary and global regularity estimates  for viscosity solutions of fully nonlinear elliptic equations whose ellipticity degenerates at the critical points of a given solution. We show that any solution is $C^{1,\alpha}$ on the boundary of the domain, for an optimal and explicit $\alpha$ given only in terms of the regularity of the boundary datum and the elliptic degeneracy degree, no matter how possibly low is the interior regularity for that class of equations. We also obtain sharp global estimates. Our findings are new even for model equations, involving only a degenerate Laplacian; all previous results of global nature give $C^{1,\alpha}$ regularity only for some small $\alpha>0$.




\end{abstract}


\section{Introduction}

The main goal of this work is to derive sharp boundary and global estimates for viscosity solutions of degenerate fully nonlinear elliptic equations modeled on
\begin{equation}\label{m-eq}
|Du|^\gamma F(D^2 u) = f(x),
\end{equation}
in a smooth bounded domain $\Omega \subset \mathbb{R}^n$, for a given Dirichlet boundary datum on a part of  $\partial \Omega$. Here $\gamma\ge0$, and the second order operator $F$ is assumed to be uniformly elliptic and Lipschitz continuous: there exist positive constants $0<\lambda \leq \Lambda$ such that
\begin{equation}\label{H}
\mathcal{M}^+_{\lambda,\Lambda}(M-N) \geq F(M)-F(N) \geq \mathcal{M}^-_{\lambda,\Lambda}(M-N),
\end{equation}
for any symmetric matrices $M,N$, where $\mathcal{M}^\pm_{\lambda,\Lambda}$ denote the classical extremal Pucci operators. In this paper we consider $C$-viscosity solutions in the sense of \cite{CIL}, all functions are considered to be continuous up to a part of $\partial\Omega$. These equations have received a lot of attention in the recent years (an extensive list of references will be given below), since they play the same role in the theory of quasi-linear elliptic equations in non-divergence form as equations based on the $p$-Laplacian do in the divergence framework. More precisely, comparing
$$
\mathrm{div}(A(x,u,Du)Du) = f(x)\quad \mbox{ vs. } \quad \mathrm{tr}(A(x,u,Du)D^2u) = f(x),
$$
we see that setting $A(x,u,Du) = |Du|^\gamma\,I$ in the former gives the $p$-Poisson equation with $p=\gamma+2$, and in the latter the equation \eqref{m-eq} with $F(D^2u)=\Delta u$. The results below are new even for $|Du|^\gamma \Delta u = f(x)$.

Equation \eqref{m-eq} degenerates along the set $\mathcal{C}(u)$ of critical points of any given solution $u\in C^1(\Omega)$. This fact has a strong effect on the  smoothness of solutions, and we cannot expect the same regularity  as in the uniformly elliptic setting. Actually, if $\gamma=0$, and $F$, $f$ are smooth, any solution is smooth too, while for $\gamma>0$ there is a natural restriction on the regularity of solutions, independently of how nice $F$ and $f$ are. Specifically, the function
$$
\xi(x) = c\, x_n^{1+ \frac{1}{1+\gamma}}
$$
is exactly $C^{1,\frac{1}{1+\gamma}}$ at $\{x_n=0\}$  and solves the boundary value problem
\begin{equation}\label{m-bvp}
\left\{
\begin{array}{rclcl}
|D\xi|^\gamma F(D^2 \xi) &=& 1&\mbox{ in }& \Omega\\
\xi&=&0&\mbox{ on }& \partial\Omega,
\end{array}
\right.
\end{equation}
for $\Omega=\rn_+=\{x_n>0\}$ and some $c=c(\gamma,F)>0$.\smallskip

The essence of our main result below is that this ``simplest" solution $\xi$ is actually the worst in terms of boundary regularity. For any uniformly elliptic $F$, any continuous $f$, and any sufficiently smooth boundary data, the solutions are at least $C^{1,\frac{1}{1+\gamma}}$ at the boundary. If in addition $F$ is convex or concave, any solution belongs to $C^{1,\frac{1}{1+\gamma}}(\overline{\Omega})$.\smallskip

The theory of equations like \eqref{m-eq} has seen important advances during the last two decades. In the pioneering works \cite{BD01}-\cite{BD04} existence and uniqueness, maximum principles, Harnack type inequalities and H\"older regularity were obtained; further important results are contained in \cite{DFQ1}-\cite{DFQ2}, \cite{I}, \cite{IS2}. A breakthrough  of \cite{IS} establishes that viscosity solutions of \eqref{m-eq} are locally $C^{1,\alpha}$ for some universal small exponent $\alpha$ depending on $\gamma$, the ellipticity and the dimension.
A  boundary extension of the result in \cite{IS}, and global regularity estimates were obtained in \cite{BD1,BD2}, see also \cite{BD3}. Specifically, in these works it is proved that solutions are $C^{1,\alpha}$ up to the boundary for some {\it small} $\alpha>0$ depending additionally on the regularity of the boundary $\partial \Omega$, as well as on the boundary data.

A natural question appears: how smooth are viscosity solutions of \eqref{m-eq}? Exact regularity measures the rate at which a solution may separate from its tangent planes (see \eqref{thmest} below) and is important for  geometric estimates and blow-up analysis, for instance in problems of free boundary type. A first observation towards the answer to this question is that, when $f\equiv 0$, viscosity solutions of \eqref{m-eq} are also viscosity solutions of $F(D^2u)=0$ (see \cite{IS}). Therefore, the optimal regularity for \eqref{m-eq} cannot be better than the optimal  regularity for solutions of $F(D^2u)=0$.  We know  (see for instance \cite{CC}) that for each $F=F(M)$ there is a  number $\alpha_F=\alpha(n,\lambda,\Lambda)>0$ such that any solution of $F(D^2u)=0$ is locally in $C^{1,\alpha_F}$. If $F$ is convex or concave then $\alpha_F=1$, by the Evans-Krylov theorem. On the other hand  the examples in \cite{NV} show  that for $n\ge5$ and any $\beta>0$ there exists $F_\beta=F_\beta(M)$ satisfying \eqref{H} such that the equation $F_\beta(D^2u)=0$ has a solution in the unit ball  which does not belong to $C^{1,\beta}$  at the center of the ball, that is,
$$
\inf \{\alpha_F \; | \, F \mbox{ satisfies } \eqref{H} \mbox{ for some } 0<\lambda\le \Lambda\} = 0.
$$

 An almost exact  {\it interior} regularity result for the degenerate equation \eqref{m-eq}, and the more general \eqref{m-geneq} below, was obtained in \cite{ART1}. Specifically, viscosity solutions  are locally in $C^{1,\sigma}$, with
\begin{equation}\label{intexp}
\sigma=\min\left\{ \tau\,, \frac{1}{1+\gamma} \right\} \; \mbox{ for any } \; \tau < \alpha_F.
\end{equation}

In this paper we address the question of optimal boundary regularity for viscosity solutions of  degenerate elliptic equations such as \eqref{m-eq}. We prove that, for a given $C^{1,\alpha'}$ boundary datum, on smooth parts of the boundary $\partial \Omega$  solutions are precisely $C^{1,\alpha}$, with
\begin{equation}\label{boundexp}
\alpha= \min \left\{\alpha\,',\frac{1}{1+\gamma} \right\}.
\end{equation}
One can think of this as follows: given a point on a smooth hypersurface on the side of which we have a solution $u$ of \eqref{m-eq}, the order of the Taylor expansion of $u$ in the normal direction is (at least) the smaller number between the order in the tangential directions and the one in the normal direction of $\xi$ given above.

Note that $\alpha$ in \eqref{boundexp} is exact and explicit, there is not a generally unknown constant such as $\alpha_F$, nor a strict inequality (compare with \eqref{intexp}). Apart from being optimal, our result shows that the interior regularity for solutions of \eqref{m-eq} has no influence on their boundary regularity. This is quite substantial, from the remarks above it follows that  for any small $\varepsilon>0$, one can exhibit  equations of the type \eqref{m-eq} which have boundary $C^{1,1-\varepsilon}$ regularity, and have solutions which are not $C^{1,\varepsilon}$ in the interior.

\smallskip

We set some notations and hypotheses to be used throughout the article. We denote $\Omega^+_r(x) := \Omega \cap B_r(x)$, $\Omega'_r(x):=\partial \Omega \cap B_r(x)$, where $B_r(x)$ is the  ball of center $x$ with radius $r>0$. In particular,  $\Omega^+_r := \Omega \cap B_r(0)$, $\Omega'_r:=\partial \Omega \cap B_r(0)$. Without restriction $0 \in \partial\Omega$. We also assume  that $\Omega$ is a $C^{2}$-domain and $f\in C(\overline{\Omega})$. All constants $C$ will be allowed to depend on $n,\gamma, \lambda, \Lambda$,  the maximal curvature of $\partial\Omega$, the assumed regularity $\alpha'$ on $\partial \Omega$, and, in the case of global estimates, the fixed exponent $\tau<\alpha_F$. 

\medskip

The following is our main result in the particular case of \eqref{m-eq}.

\begin{thm} \label{mainthm}
Suppose $F$ satisfies \eqref{H}.  Let $u$ be a viscosity solution of \eqref{m-eq} in $\Omega$, such that the restriction $g = u|_{\partial\Omega} \in C^{1,\alpha\,'}(\Omega'_1)$, for some $0<\alpha\,'<1$. Then for
\begin{equation}\label{exp}
\alpha=\min\left\{\alpha\,', \frac{1}{1+\gamma}\right\}, \qquad \beta= \min\{\alpha,\tau\}\;\mbox{ with } \: \tau<\alpha_F,
\end{equation} and for some constant $C$ we have
\begin{equation}\label{thmest}
|u(x)-u(x_0)-Du(x_0)\cdot (x-x_0)| \leq C A \,|x-x_0|^{1+\alpha}, \quad \mbox{ for }x\in \Omega^+_{1/2},\; x_0 \in \Omega'_{1/2},
\end{equation}
as well as
\begin{equation}\label{thmest2}
\|u\|_{C^{1,\beta}(\Omega^+_{1/2})}\le CA,
\end{equation}
where $A=\|u\|_{L^\infty(\Omega_1^+)}+\|g\|_{C^{1,\alpha\,'}(\Omega_1')} +\|f\|_{L^\infty(\Omega_1^+)}^{\frac{1}{1+\gamma}}$.
\end{thm}

Next, we give a more general version of this theorem, for the degenerate equation
\begin{equation}\label{m-geneq}
H(|Du|,x) F(D^2 u,x) = f(x) \quad \mbox{in } \; \Omega
\end{equation}
(see also Remark \ref{norestr} below), where $F$ and $H$ are continuous, and for some  modulus of continuity $\omega(s)$
 \begin{equation}\label{HH}
\left\{
\begin{array}{c}
F(\cdot,x) \mbox{ satisfies } \eqref{H},\qquad |F(M,x)-F(M,y)| \leq \omega(|x-y|)|M|, \\
\lambda |p|^\gamma\le H(p,x)\le \Lambda |p|^\gamma, \qquad M,N \in Sym(\mathbb{R}^n), p\in\mathbb{R}^n, x,y\in \overline{\Omega}.
\end{array}
\right.
\end{equation}

We say that \eqref{m-geneq} admits global $C^{1,\bar\omega}$-estimates if each viscosity solution of \eqref{m-geneq} in $\Omega$ with a $C^{1,\alpha'}$ boundary datum  is $C^1$ up to the boundary and \eqref{thmest} holds with $|x-x_0|^{1+\alpha}$ replaced by $|x-x_0|\bar\omega(|x-x_0|)$, for $x,x_0\in \Omega^+_{1}$.

\begin{thm}\label{genthm}
Assume \eqref{HH} and that \eqref{m-geneq} admits global $C^{1,\bar\omega}$-estimates, for some modulus~$\bar\omega(s)$. Let $u$ be a viscosity solution of \eqref{m-geneq}, with $g = u|_{\partial\Omega} \in C^{1,\alpha\,'}(\Omega'_1)$,  $0<\alpha\,'<1$. Then \eqref{exp}-\eqref{thmest2} hold, with $\alpha_F=\inf_{x_0\in\overline{\Omega}}\alpha_{F(\cdot,x_0)}$, and  $C$  depending also on $\omega$, $\bar\omega$.
\end{thm}

\begin{preremark}\label{norestr} Note it is not really a restriction that the operator $F$ does not depend on $Du$ and $u$ since, as  global $C^1$-regularity is already available, first and zero order terms can be incorporated in the $x$-dependence.\end{preremark}

\begin{preremark}\label{genop} Theorem \ref{mainthm} is a consequence of Theorem \ref{genthm} and \cite{BD1, BD2}, where it is proved that \eqref{m-eq}, even with  $F(D^2u)$ replaced by $F(D^2u) + b(x).Du$ or $F(D^2u) + b(x)|Du|^a$, $a\le 1$, admits global $C^{1,\alpha_0}$-estimates, for some small $\alpha_0>0$. Thus Theorem \ref{mainthm} is also valid for $F(D^2u)$ replaced by $F(D^2u) + h(x).Du$, incorporating  $|Du|^\gamma h(x).Du$ into $f(x)$. \end{preremark}

\begin{preremark}\label{remess} Theorem \ref{genthm} is stated in a way to emphasize an essential feature of the method: that ``minimal" $C^{1+}$-regularity implies the ``maximal" $C^{1,\alpha}$-regularity for degenerate equations. We expect \eqref{HH} suffices to ensure that \eqref{m-geneq} admits global $C^{1,\alpha_0}$-estimates, for some small $\alpha_0>0$ (as in \cite{BD1}, \cite{BD2}).\end{preremark}

Theorem \ref{genthm} for $\gamma=0$ gives optimal boundary and global $C^{1,\alpha}$-regularity for uniformly elliptic equations, stating that their solutions are as regular at $\partial\Omega$ as the boundary datum itself. This can be proved by Caffarelli's perturbation method \cite{C}, \cite{T}, more precisely, by seeing the problem as ``tangential" to
\cite[Lemma 4.1]{SS}. Results of similar vein, even in larger generality (for unbounded coefficients and $L^p$-viscosity solutions), have recently appeared in \cite{BGMW} for the so-called uniformly elliptic $S^*$-class where the optimal boundary exponent depends also on $\alpha_F$, as well as in the recent preprints \cite{LWZ}, \cite{NS}.
We thus assume Theorem~\ref{genthm} for $\gamma=0$ might be known to the experts; however, since this particular case is interesting in itself and deserves a quotable source, and since we use it in the proof of the full Theorem \ref{genthm} and strive to make this work self-contained, we give a complete proof in Section \ref{unifcase}.

\medskip

The main tools used in the proof of Theorem \ref{genthm} are the gradient oscillation estimates for degenerate elliptic equations established in \cite{ATU2D,ATUtow}, and the boundary regularity estimates for uniformly elliptic equations from \cite{SS}. Our analysis consists in controlling the gradient at a boundary point, imposing a precise sense on how that point is close to the critical set $\mathcal{C}(u)$, see Theorem \ref{smallgrad} in Section \ref{degcase}. If the gradient does not have that prescribed behaviour, in Section \ref{finalproof} we show that the boundary point is quantifiably far from $\mathcal{C}(u)$ so after a rescaling which depends on the value of $|Du|$ at that point, the equation \eqref{m-geneq} becomes uniformly elliptic in a domain of fixed size, and we are able to apply the optimal regularity estimates for uniformly elliptic equations given in Section \ref{unifcase}. An additional difficulty appears in the proof of the global estimate \eqref{thmest2}, since we need to compare the size of the gradient with a power of the distance to the boundary, at points where the latter is small.

\section{Boundary regularity for uniformly elliptic equations}\label{unifcase}

In this section we consider viscosity solutions of the uniformly elliptic equation \eqref{m-geneq} with $H=1$ or $\gamma=0$, that is,
\begin{equation}\label{unifeq}
\begin{array}{rcc}
F(D^2 u, x)= f & \mbox{in} & \Omega_1^+ \\[0.1cm]
u=g & \mbox{on} &\; \Omega_1',
\end{array}
\end{equation}
and show the gradient of $u$ at the boundary is H\"older continuous with the same exponent as the gradient of $g$.  This fact plays a role in the proof of Theorem \ref{genthm}. It is also interesting in itself, showing for instance that one cannot construct an example of low regularity solutions such as the ones in \cite{NV}, in which the regularity of the solution is lower in one direction than in the other $n-1$ directions.

We recall that if $F$ depends also on $Du$ and $u$ and $C^1$-estimates for the Dirichlet problem are available, this problem writes in the form \eqref{unifeq}, since $Du$ and $u$ can be incorporated in the $x$-dependence. This is true for instance for operators with up-to-quadratic growth in the gradient such as the ones considered in \cite{Sk}, \cite{N}. Below we will use the $C^1$ estimates from \cite{W} in order to make simplifications in the proof, since we do not aim at maximal generality here (see also Remark \ref{remnotgen} below). We will also simplify by flattening the boundary upfront, even though the result is  pointwise in nature since a general  boundary would flatten around each fixed point in the blow-up process (as in \eqref{rescdom} below).

The following theorem is the main result of this section. Here universal constants depend on $n,\lambda,\Lambda,\alpha', \omega$, and the maximal curvature of $\partial \Omega$.

\begin{thm}\label{unifthm}
Suppose $\Omega$ is a $C^2$-domain. Let $u$ be a viscosity solution to the equation \eqref{unifeq} where $F$ satisfies \eqref{HH},  and $g \in C^{1,\alpha\,'}(\Omega'_1)$, for some $0<\alpha\,'<1$. Then for some universal constant $C$
\begin{equation}\label{unifthmest}
|u(x)-u(x_0)-Du(x_0)\cdot (x-x_0)| \leq C A \,|x-x_0|^{1+\alpha'}, \quad \mbox{ for } \; x\in \Omega^+_{1/2},\;x_0 \in \Omega'_{1/2},
\end{equation}
and for any $\tau < \alpha_F$ there exists a universal constant $\overline C$, such that
\begin{equation} \label{unifthmest2}
\|u\|_{C^{1,\overline\beta}(\Omega^+_{1/2})}\le \overline CA, \quad \mbox{where }\; \overline\beta= \min\{\alpha',\tau\},
\end{equation}
 and $A=\|u\|_{L^\infty(\Omega_1^+)}+\|g\|_{C^{1,\alpha\,'}(\Omega_1')} +\|f\|_{L^\infty(\Omega_1^+)}$.
\end{thm}

\begin{remark}\label{remnotgen} This theorem is stated only in the generality that we need for the subsequent developments for degenerate equations. Much more general versions of this result can be envisioned, featuring for instance VMO-like second-order coefficients instead of continuous, unbounded measurable first- and zero-order coefficients, $L^p$-viscosity solutions, power growth in the gradient, less regularity of the boundary, $C^{1, Dini}$-regularity type results. We refer to \cite{BGMW}, \cite{LWZ}, \cite{NS},  and the extensive list of references in these works for similar regularity results of more general nature.
\end{remark}

\begin{proof}[Proof of Theorem \ref{unifthm}] From the global results in \cite{W} we know that $Du$ is continuous and uniformly bounded up to the boundary. It is standard that for any point $x_0 \in \partial\Omega$ there exists a $C^2$ diffeomorphism $\varphi: \Omega\to B_1^+ $ with $\varphi(x_0)=0$, such that the function $v=u\circ \varphi^{-1}$ is a viscosity solution for $F_{\varphi}(D^2 v, x)=f(\varphi^{-1}(x)) \; \mbox{ in } \; B_1^+$,
with $F_\varphi$ satisfying \eqref{HH} with (possibly) different $\lambda, \Lambda, $ and $\omega$ depending only on those for $F$, the $C^1$-norm of $u$ and the $C^2$-norm of $\partial\Omega$. Therefore, we may only consider the special case $\Omega_1^+ = B_1^+$, $\Omega_1'=B_1'$, denoting $B_r^+=B_r \cap \{x_n>0\}$, $B_r'=B_r \cap \{x_n=0\}$.

In addition, translating the origin, we assume  $x_0=0$. Since $u(x)-u(0)-Du(0)\cdot x$ satisfies the same equation as $u(x)$, we can assume  $u(0)=|Du(0)|=0$.  So the first estimate in \eqref{unifthmest} follows from
\begin{equation}\label{unifest1}
\sup\limits_{x \in B^+_\rho}|u(x)| \leq CA\, \rho^{1+\alpha'}
\end{equation}
for $\rho\in(0,\rho_0)$ and any viscosity solutions $u$ of \eqref{unifeq} in $B_1^+$ satisfying $u(0)=|Du(0)|=0$, where $\rho_0$ is universal (by a trivial covering argument).

\begin{lemma}\label{complem}
There exists $\varrho >0$ depending only on $n,\lambda,\Lambda,\omega$ and $\alpha'$ such that for each $0<\rho \leq \varrho$ we can choose $\delta$ depending on $\rho$ such that if
\begin{equation}\label{small}
\|u\|_{L^\infty(B_1^+)} \leq 1, \quad \sup\limits_{0<s\leq 1}\omega(s) \leq \delta, \quad \|f\|_{L^\infty(B_1^+)}\leq \delta \quad \mbox{and} \quad \|g\|_{L^\infty(B_1')} \leq \delta,
\end{equation}
and $u$ is a viscosity solution of \eqref{unifeq} in $\Omega=B_1^+$ under conditions \eqref{HH}, such that $u(0)=|Du(0)|=0$, then
$$
\sup\limits_{B_\rho^+}|u(x)| \leq \rho^{1+\alpha'}.
$$
\end{lemma}

\begin{proof}
Note $\delta=0$ means we have an equation with constant coefficients and a solution which vanishes on a flat part of the boundary. This situation  was treated in \cite[Lemma 4.1]{SS} where it is shown that any viscosity solution $v$ of
\begin{equation}\label{limeq1}
F(D^2 v, 0)= 0\;  \mbox{ in } \; B_{3/4}^+ \qquad
v=0\; \mbox{ on }\;  B_{3/4}' \qquad
|Dv(0)|=0,
\end{equation}
is such that for some  $\tilde{C}>1$ depending on $n,\lambda,\Lambda$, and for all $\rho\in (0,1/2)$
\begin{equation}\label{target}
 \sup_{B_\rho^+}|v(x)| \leq \tilde{C}\,\rho^{2}.
\end{equation}

We set $\varrho=\left(2\tilde{C} \right)^{-\frac{1}{1-\alpha'}}$. If Lemma \ref{complem} fails for some $\rho\le \varrho$, we can find  sequences $u_k$, $F_k$, $\omega_k$ and $f_k$ such that $u_k$ satisfies $\|u_k\|_{L^\infty(B_1^+)} \leq 1$, $u_k(0)=|Du_k(0)|=0$ and, in the viscosity sense,
$$
F_k(D^2 u_k, x)= f_k  \quad \mbox{in} \quad B_1^+,
$$
where
$$
 \sup\limits_{0<s\leq 1}\omega_k(s) +\|f_k\|_{L^\infty(B_1^+)}+ \|u_k\|_{L^\infty(B_1')} \leq \frac 1k,
$$
but
\begin{equation}\label{contradiction}
\sup\limits_{B_\rho^+}|u_k(x)| > \rho^{\,1+\alpha'}.
\end{equation}
By the $C^{1+}$ global regularity estimates of \cite{W}, up to a subsequence we have that $u_k$ and $Du_k$  converge uniformly in $\overline{B^+_{3/4}}$ to some function $u_\infty$, resp. $Du_\infty$. By the stability properties of viscosity solutions $u_\infty$ solves \eqref{limeq1}, for some uniformly elliptic operator $F_\infty(D^2u)$ which is a limit of a subsequence of $F_k$.
Hence \eqref{target} holds for $v=u_\infty$, so
$$
\sup\limits_{B_\rho^+}|u_\infty(x)| \leq \tilde{C} \rho^2 \leq \frac{1}{2}\rho^{1+\alpha'}.
$$
Hence, we can find $\overline{k}$ depending on $\rho$ such that for $k \geq \overline{k}$,
\begin{equation}
|u_k(x)|  \leq  |u_\infty(x)|+|u_k(x)-u_\infty(x)| \leq \dfrac{1}{2}\rho^{1+\alpha'}+\dfrac{1}{2}\rho^{1+\alpha'},
\end{equation}
 a contradiction with \eqref{contradiction}.
\end{proof}

\medskip

\begin{prop}\label{kstepUnif}
 There exist universal parameters $0<\varrho<1$ and $\varpi>0$ such that if
\begin{equation}\label{smallmod}
\|u\|_{L^\infty(B_1^+)}\leq 1,\quad \sup\limits_{0<s\leq 1}\omega(s)\leq \varpi, \quad \|f\|_{L^\infty(B_1^+)} \leq \varpi, \quad \mbox{and} \quad \|g\|_{C^{1,\alpha\,'}(B_1')} \leq \varpi,
\end{equation}
and $u$ is a viscosity solution of \eqref{unifeq} in $\Omega=B_1^+$ under conditions \eqref{HH}, such that $u(0)=|Du(0)|=0$, then
for each $0< r \leq {\varrho}$,
\begin{equation}\label{finalest}
\sup\limits_{B_{r}}|u(x)|    \leq  C r^{1+\alpha\,'}, \quad \mbox{ with }\; C=  \varrho^{-(1+\alpha\,')}.
\end{equation}
\end{prop}

\begin{proof} First, for
 $\varrho$ and $\varpi=\delta(\varrho)$ as in Lemma \ref{complem}, we shall prove inductively that for each integer $k\ge0$
\begin{equation}\label{kcontrol}
\sup\limits_{B_{\varrho^{k}}^+}|u(x)| \leq \varrho^{k(1+\alpha\,')}.
\end{equation}
 When $k=0$ estimate \eqref{kcontrol} follows directly by
$\|u\|_{L^\infty(B_1^+)}\leq 1$.
We assume \eqref{kcontrol} holds for some value $k=j$, and define
$$
\tilde u(x):=  \frac{u(\varrho^{\,j} x)}{\varrho^{\,j(1+\alpha\,')}}.
$$
Note that $\tilde u$ is a viscosity solution for $\tilde F(D^2 \tilde u,x)=\tilde f$ in $B_1^+$, where
$$
\tilde F(M,x):= \varrho^{j(1-\alpha')}F \left( \varrho^{-j(1-\alpha')} M, \varrho^j x \right) \quad \mbox{and} \quad \tilde f(x):=\varrho^{j(1-\alpha')}f(\varrho^j x).
$$
We observe that the operator $\tilde F$ has the same ellipticity constants as  $F$, and $\|\tilde f\|_{L^\infty(B_1^+)} \leq \| f\|_{L^\infty(B_1^+)}\leq \varpi$. Denoting by $\tilde\omega$ the modulus of continuity related to $\tilde F$, we have
$$
\sup\limits_{0<s\leq 1}\tilde \omega(s) \leq \sup\limits_{0<s\leq 1}\omega(\varrho^{\,j}s) \leq \varpi.
$$
Also $\tilde u(0)=|D\tilde u(0)|=0$, and, denoting by $\tilde g$ the restriction of $\tilde u$ to $B_1'$, we have
$$
\|\tilde g\|_{L^\infty (B_1')} \leq \|g\|_{C^{1,\alpha\,'}(B_1')} \leq \varpi,
$$
since $g(0)=|Dg(0)|=0$. Finally, \eqref{kcontrol} holds for $k=j$, so
$$
\sup_{B_1^+}|\tilde u| \leq 1.
$$
From the last three estimates and  Lemma \ref{complem} applied to $\tilde F(D^2 \tilde u,x)=\tilde f$ in $B_1^+$  we infer
$$
\sup\limits_{B_\varrho^+}|\tilde u(x)| \leq \varrho^{1+\alpha\,'},
$$
which means \eqref{kcontrol} is true for  $k=j+1$, and proves \eqref{kcontrol} for any integer $k>0$.

 Let now $0< r \leq {\varrho}$, and $k$ be the integer such that ${\varrho}^{\,k+1}<r\leq {\varrho}^{\,k}$. From \eqref{kcontrol}
$$
\sup\limits_{B_{r}}|u(x)|  \leq  \sup\limits_{B_{\varrho^k}}|u(x)|  \leq  C r^{1+\alpha\,'}, \quad \mbox{ for }\; C=  \varrho^{-(1+\alpha\,')},
$$
and estimate \eqref{finalest} is obtained.
\end{proof}

\smallskip

Now we end the proof of Theorem \ref{unifthm}.
 We rescale $u$ by setting
$$
\overline{u}(x):= \kappa u(\tau x),
$$
for
$$
 \kappa= \frac{1}{\|u\|_{L^{\infty}(B_1^+)}+\varpi^{-1}\|g\|_{C^{1,\alpha\,'}(B_1')}+ \varpi^{-1}\| f\|_{L^\infty(B_1^+)}},
$$
and $0<\tau<1$ sufficiently small that $$\sup\limits_{0<s\leq \tau}\omega(s)\le\varpi.$$
Note that $\overline F(D^2\overline u,x)=\overline f(x)$ in $B_{1}^+$, for
$$
\overline F(M,x):=\kappa \tau^2\cdot F([\kappa\tau^2]^{-1} \cdot M, \tau x) \quad \mbox{and} \quad \overline f(x):=\kappa \tau^2 f(\tau x)
$$
We easily check that we can  apply Proposition \ref{kstepUnif} to $\overline F$ and $\overline u$. As a direct consequence, we obtain \eqref{finalest} for $\overline u$, hence \eqref{unifest1}  for $u$. Estimate \eqref{unifthmest} is proved.

It is well known that to obtain \eqref{unifthmest2} it is sufficient to prove that
\begin{equation}\label{taylor1}
|u(y)-u(x)-Du(x)\cdot (y-x)| \leq C A \,|y-x|^{1+\overline\beta}, \qquad x,y\in {\Omega_{1/2}^+}
\end{equation}
(see for instance the proofs in the appendices of \cite{BFM}, \cite{BGMW}).

We recall that by the celebrated \cite[Theorem 2]{C}, if $F(D^2u(y),y) = f(y)$ in $B_1(x)$, then for any $\tau<\alpha_F$ we have the interior estimate
\begin{equation}\label{taylor2}
|u(y)-u(x)-Du(x)\cdot (y-x)| \leq \overline{C} A_0 \,|y-x|^{1+\tau}, \qquad y\in {B_{1/2}(x)}.
\end{equation}
where $A_0=\|u\|_{L^\infty(B_1(x))} +\|f\|_{L^\infty(B_1(x))}$.

How to combine the boundary estimate \eqref{unifthmest} and the interior estimate \eqref{taylor2} into the global estimate \eqref{taylor1} is also well-known -- see Remark 10.2 and the proof of Proposition 2.4 in the appendix of \cite{MS}.
\end{proof}

\section{Boundary gradient estimates for the degenerate case} \label{degcase}

From now on we work under the hypotheses of Theorem \ref{genthm}.
We recall $u$ is a viscosity solution for \eqref{m-geneq} with $u|_{\Omega_1'}=g \in C^{1,\alpha\,'}(\Omega_1')$. By virtue of the assumed $C^{1,\bar\omega}$ global regularity estimates,  the derivatives of $u$ are continuous and uniformly bounded in $\Omega_1^+$. More precisely, for $x,y \in {\Omega_1^+}$
\begin{equation}\label{C1omega}
\left\{\begin{array}{rcl}
|u(x)-u(y)-Du(y)\cdot (x-y)| &\leq& CA|x-y|\bar\omega(|x-y|) \\
|Du(x)-Du(y)| &\leq& CA\,\bar\omega(|x-y|), \quad |Du(x)| \leq CA,
\end{array}
\right.
\end{equation}
for $A$ as in Theorem \ref{mainthm}.

\begin{remark}
We recall that for \eqref{m-eq}, and more generally for operators as in Remark \ref{genop}, it is proved in \cite[Theorem 1.1]{BD2} that \eqref{C1omega} holds with $\bar\omega(s)=s^{\beta_0}$ for some small universal $\beta_0>0$. We remark it is stated in that theorem that $\beta_0$ may depend also on the norm of the first order coefficient $b(x)$, however this dependence can be eliminated as follows: examining the proof of \cite[Theorem 1.1]{BD1} one sees that it produces the same $\beta_0$ for all $b$ with $\|b\|_\infty\le1$. Then rescaling $\tilde u(x) = u(x/\|b\|_\infty)$ leads to an equation for $\tilde u$ in which the coefficient $\tilde b$ has norm smaller than $1$; so only $C$ depends on $\|b\|_\infty$.
\end{remark}

 In order to state the results in this section, we denote with $[Dg]_{C^{\,0,\alpha\,'}(\Omega_1')}$ the $\alpha\,'$-H\"older seminorm of $Dg$ restricted to $\Omega'_1$. For simplicity, we always suppose $x_0=0 \in \partial\Omega$ and $u(0)=0$ (replacing $u$ by $u-u(0)$ in the equation). Constants will be called universal if they depend on $n,\gamma, \lambda, \Lambda$, $\alpha'$, $\omega$, $\bar \omega$, as well as the maximal curvature of $\partial\Omega$.

\medskip

Here is the main result of this section.

\begin{thm}\label{smallgrad} Under the hypotheses of Theorem \ref{genthm}, let $u$ satisfy \eqref{m-geneq}, $u(0)=0$.
There exist small positive universal numbers $\rho_0$ and $\delta_0$, such that for $\alpha$ as in \eqref{boundexp},  if
\begin{equation}\label{boundcond1}
|Du(0)| \leq \delta_0 \cdot \rho^{\alpha},
\end{equation}
holds for some $0<\rho \leq \rho_0$ then
\begin{equation}\label{boundres1}
\sup_{\Omega_{\rho}^+} |u(x)|  \leq CA\,\rho^{1+\alpha}.
\end{equation}
\end{thm}

First, we need the following approximation lemma, whose proof is similar to that of Lemma \ref{complem}, with smallness assumptions on the gradient of the solution.

\begin{lemma}\label{complemmagrad}
Under the hypotheses of Theorem \ref{genthm}, let $u$ satisfy \eqref{m-geneq}, $u(0)=0$, and $\|u\|_{L^\infty(\Omega)} \leq 1$. There exists a universal positive number $\rho_0$ such that for every $0<\rho \leq \rho_0$ we can find $\delta>0$ for which
\begin{equation}\label{estcond1}
\sup\limits_{0<s\leq 1}\omega(s)+ \|f\|_{L^\infty(\Omega^+_1)}+[Dg]_{C^{\,0,\alpha\,'}(\Omega_1')}+|Du(0)| \leq \delta
\end{equation}
implies
\begin{equation}\label{osceq1}
\sup_{\Omega_{\rho}^+} |u(x)|  \leq \rho^{1+\alpha}.
\end{equation}
\end{lemma}

\begin{proof}
 We assume for contradiction that  for some $\rho>0$ there exist sequences of continuous functions $H_k(p,x)$, $F_k(M,x)$ satisfying $\eqref{HH}$, $f_k \in C(\overline{\Omega})$, $u_k \in C^{1,\bar\omega}(\overline{\Omega})$ satisfying \eqref{C1omega}, with $u_k(0)=0$, $\|u_k\|_{L^{\infty}(\Omega)} \leq 1$, such that $u_k$ solves
\begin{equation}\label{geneqk}
H_k(|Du_k|,x) F_k(D^2u_k,x)=f_k
 \quad \mbox{in} \quad \Omega_1^+,
\end{equation}
and for each integer $k>0$ we have
\begin{equation}\label{estcondk}
\sup\limits_{0<s\leq 1}\omega_k(s)+\|f_k\|_{L^\infty(\Omega^+_1)}+[Du_k]_{C^{\,0,\alpha'}(\Omega_1')}+|Du_k(0)|  \leq \frac 1k,
\end{equation}
but
\begin{equation}\label{estcontr}
\sup_{\Omega_{\rho}^+} |u_k(x)|  > {\rho}^{\,1+\alpha'}.
\end{equation}

Subsequences of $H_k$ and $F_k$ converge locally uniformly to some functions $H_\infty(p,x)$ and $F_\infty(M)$ (the latter is independent of $x$ since $\omega_k\to0$ by \eqref{estcondk}), which satisfy \eqref{HH}.
Since $u_k$ is globally $C^1$-equicontinous in $\overline{\Omega}$, Arzela-Ascoli theorem provides a subsequence of $u_k$ which converges in $C^1$ to a function $u_\infty$. By using \eqref{geneqk}, \eqref{estcondk} together with $u_k(0)=0$, we obtain by stability properties of viscosity solutions that $u_\infty$ solves in the viscosity sense the following problem
$$
\begin{array}{rcc}
H(|Du_\infty|,x) F_\infty(D^2 u_\infty)=0 \, & \mbox{in} & \Omega_{3/4}^+ \\[0.16cm]
u_\infty=0 \, & \mbox{on} & \Omega_{3/4}' \\[0.16cm]
|Du_\infty(0)|=0. & &
\end{array}
$$
 Hence by \cite[Lemma 6]{IS} or Lemma \ref{genmult} below, we observe that in fact $u_\infty$ solves
$$
F_\infty(D^2 u_\infty)=0\,  \mbox{ in }  \Omega_{3/4}^+, \qquad
u_\infty=0\,  \mbox{ on }  \Omega_{3/4}', \qquad
|Du_\infty(0)|=0.
$$
Consequently, \cite[Lemma 4.1]{SS} provides that for some  $C_0=C_0(n,\lambda,\Lambda)$ and all $0<\rho\le\rho_0$
$$
\sup\limits_{\Omega^+_{ \rho}}|u_\infty(x)| \leq C_0 \,\rho^2 \leq \frac{1}{2}\rho^{\,1+\alpha'},
$$
where the latter inequality determines our choice of $\rho_0$, namely  $\rho_0 = (2C_0)^{\,-\frac{1}{1-\alpha'}}$. Therefore, we conclude that for $k$ large (depending on $\rho$)
$$
\sup\limits_{\Omega^+_{\rho}}|u_k(x)|  \leq \sup\limits_{\Omega^+_{\rho}}|u_k(x)-u_\infty(x)|+\sup\limits_{\Omega^+_{\rho}}|u_\infty(x)| \leq  \dfrac 12 \rho^{\,1+\alpha'}+\dfrac 12 \rho^{\,1+\alpha'} ,
$$
 a contradiction.
\end{proof}

For further reference we record the following trivial extension of \cite[Lemma 6]{IS}.
\begin{lemma}\label{genmult} Assume $H(p,x)$ and $F(M,x)$ are continuous and satisfy \eqref{HH}. If $u$ is a viscosity solution of $H(Du+q,x)F(D^2u,x)=0$ for some $q\in\mathbb{R}^n$, then $u$ is a viscosity solution of $F(D^2u,x)=0$.
\end{lemma}
\begin{proof} We repeat the proof of \cite[Lemma 6]{IS}. With the notations of that proof, we only need to reduce the ball $B_r$ to be small enough that $F(A,0)<0$ implies $F(A,x)<0$ for $x\in B_r$.
\end{proof}
\begin{prop}\label{kstep}
Under the hypotheses of Theorem \ref{genthm}, let $u$ satisfy \eqref{m-geneq}, $u(0)=0$, and $\|u\|_{L^\infty(\Omega)} \leq 1$. Let $\alpha$ be as in \eqref{boundexp}. There exist small positive universal numbers $\rho_0$ and $\delta_0$ such that if
\begin{equation}\label{estcond2}
\sup\limits_{0<s\leq 1}\omega(s)+ \|f\|_{L^\infty(\Omega^+_1)}+[Dg]_{C^{\,0,\alpha\,'}(\Omega_1')} \leq \delta_0,
\end{equation}
and for some $k\in\mathbb{N}$
\begin{equation}\label{boundcond2}
|Du(0)| \leq \delta_0 \cdot \rho_0^{\, k\alpha},
\end{equation}
 then
\begin{equation}\label{discrete}
\sup_{\Omega_{\rho_0^k}^+} |u(x)|  \leq \rho_0^{\,k(1+\alpha)}.
\end{equation}
\end{prop}

\begin{proof} We choose $\rho_0<2^{-1/\alpha'}$ and $\delta_0=\delta(\rho_0)/2$, the numbers given by Lemma \ref{complemmagrad}.

Assuming  \eqref{estcond2}, we shall prove inductively that, for each positive integer $k$, estimate \eqref{boundcond2} implies that estimate \eqref{discrete} holds. The case $k=1$ follows directly from Lemma \ref{complemmagrad}.

Let us suppose that  \eqref{boundcond2} implies \eqref{discrete} for $k=j$. We claim the same is true for $k=j+1$. Indeed, supposing that \eqref{boundcond2} holds for $k=j+1$, we argue as follows: denote $r_j:=\rho_0^{\,j}$ and
$$
u_j(x):=\frac{u(r_j\, x)}{r_j^{1+\alpha}} \quad \mbox{in} \quad \Omega_{1,j}^+,
$$
where, assuming without loss that in a neighborhood of $0$ (which can be supposed to be of size 1, after a rescaling depending only on the curvature of $\Omega$) the domain $\Omega$ can be represented as
$$
\partial \Omega= \left\{ x=(x',x_n)\in \mathbb{R}^n\::\: x_n = a(x')\right\}, \quad  \Omega= \left\{ x=(x',x_n)\in \mathbb{R}^n\::\: x_n > a(x')\right\},
$$
for some $C^2$ smooth function $a$ defined in a neighborhood of $0\in\mathbb{R}^{d-1}$, we have set
\begin{equation}\label{rescdom}
\Omega_{1,r_j}^+=\Omega_{1,j}^+ =\left\{ x\in \mathbb{R}^n\::\: |x|<1,\:x_n > a_j(x')\right\},  \quad\mbox{ with }\quad  a_j(x') = \frac{a(r_j x')}{r_j}.
\end{equation}
We observe that $|Da_j|$ and $|D^2a_j|$ are bounded only in terms of $|Da|$ and $|D^2a|$. Note that $\partial\Omega_{1,j}^+\cap \{x_n=a_j(x')\}$ actually ``flattens" as $j$ increases.

We see  that $u_j$ solves
\begin{equation}\label{equal1}
H_j(|Du_j|, x) F_j(D^2u_j,x)=f_j(x)\quad\mbox{ in }\;\Omega_{1,j}^+,
\end{equation}
 where
$$
H_j(p,x) =r_j^{-\alpha\gamma}H(r_j^\alpha p,r_jx),\quad F_j(M,x):= r_j^{1-\alpha}F(r_j^{\alpha-1}M, r_jx), \quad f_j(x):=r_j^{(1-\alpha)-\alpha\gamma}f(r_j\, x).
$$
Also, we observe that $H_j$, $F_j$ satisfy \eqref{HH} with the same constants, and the modulus $\omega_j(s) = \omega(r_js)$, so $$\sup\limits_{0<s\leq 1}\omega_j(s)\le \sup\limits_{0<s\leq 1}\omega(s).$$
Taking into account that \eqref{discrete} holds for $k=j$, we  obtain
$$
\|u_j\|_{L^{\infty}(\Omega_{1,j}^+)} \leq 1.
$$
Since $\alpha \leq \min\{\alpha\,',(1+\gamma)^{-1}\}$ we easily check that \begin{equation}\label{contr}
\|f_j\|_{L^{\infty}(\Omega_{1,j}^+)} \leq \|f\|_{L^{\infty}(\Omega_{1}^+)}   \quad \mbox{ and }\quad [Du_j]_{C^{0,\alpha'}(\Omega_{1,j}')} \leq [Du]_{C^{0,\alpha'}(\Omega_1')}=[Dg]_{C^{0,\alpha'}(\Omega_{1}')}.
\end{equation}
Additionally, since \eqref{boundcond2} holds for $k=j+1$, we have
$$
|Du_j(0)| \leq \delta_0 \cdot \rho_0^{\alpha}.
$$
Hence $u_j$ satisfies the assumptions of Lemma \ref{complemmagrad}, in particular  \eqref{estcond1} for $\delta=\delta_0$, and from that lemma applied to \eqref{equal1} we obtain
$$
\sup_{\Omega_{\rho_0,j}^+} |u_j(x)|  \leq \rho_0^{1+\alpha}.
$$
This means $u$ satisfies \eqref{discrete} for $k=j+1$.
\end{proof}

Now, we prove the main result of this section.

\begin{proof}[Proof of Theorem \ref{smallgrad}] Let $\rho_0$, $\delta_0$ be chosen as in the previous proposition.

First,  assume $\|u\|_{L^\infty(\Omega)}\leq 1$ and the smallness conditions \eqref{estcond2}, which permit us to use Proposition \ref{kstep}. Given $0<\rho \leq \rho_0$, let $k$ be the positive integer such that $\rho_0^{k+1}< \rho \leq \rho_0^{k}$. Then, if
$$
|Du(0)| \leq \delta_0 \rho^{\alpha},
$$
which implies $|Du(0)|\leq \delta_0\,\rho_0^{k\alpha}$, we have by Proposition \ref{kstep}
\begin{equation}\label{normest}
\sup_{B_{\rho}^+} |u(x)|  \leq \rho_0^{k(1+\alpha)} \leq C\rho^{1+\alpha},\qquad C=\rho_0^{-(1+\alpha)},
\end{equation}
so Theorem \ref{smallgrad} is proved under the additional smallness assumption.

In the general case, for $\delta_0$ as in \eqref{estcond2}, we choose two constants $\tau, \kappa\in (0,1)$ as follows: $\tau$ is so small that
$$
\sup\limits_{0<s\leq \tau}\omega(s)\le\delta_0/2,$$
and
$$ \kappa = \frac{\delta_0/2}{\|u\|_{L^{\infty}(\Omega_1^+)}+\tau^{1+\alpha'}[Dg]_{C^{0,\alpha\,'}(\Omega_1')} + \tau^{\frac{2+\gamma}{1+\gamma}}\|f\|_{L^{\infty}(\Omega_1^+)}^{\frac{1}{1+\gamma}} }.
$$
We define the  function
\begin{equation}\label{normsol}
\tilde{u}(x):= \kappa \cdot u(\tau x)
\end{equation}
which solves
\begin{equation}\label{normeq}
\tilde{H}(D\tilde{u},x)\tilde F(D^2 \tilde u,x)=\tilde f\quad\mbox{ in }\; \Omega^+_{1,\tau},
\end{equation}
 for
\begin{equation}\label{normcoef}
\begin{cases}
\tilde{H}(p,x) = (\kappa\tau)^{\gamma}H(p/(\kappa\tau), \tau x), \qquad\tilde{F}(M,x):=\kappa \tau^2\cdot F([\kappa\tau^2]^{-1} \cdot M, \tau x), \\
 \tilde{f}(s):=\kappa^{1+\gamma}\tau^{2+\gamma} f (\tau x),\qquad \qquad \Omega^+_{1,\tau}  =\left\{ x\in \mathbb{R}^n\::\: |x|<1,\:x_n > \frac{a(\tau x)}{\tau}\right\},
\end{cases}
\end{equation}
We easily check that
$$
\tilde u(0)=0, \quad \|\tilde{u}\|_{L^\infty(\Omega^+_{1,\tau})} \leq 1, \quad |D\tilde u(0)|\le |Du(0)|\le \delta_0 \rho^{\alpha}, \quad \mbox{and}
$$
$$
\sup\limits_{0<s\leq 1}\tilde{\omega}(s)+\|\tilde{f}\|_{L^\infty(\Omega^+_{1,\tau})}+
[D\tilde{u}]_{C^{\,0,\alpha\,'}(\Omega^\prime_{1,\tau})} \leq \delta_0.
$$
By applying estimate \eqref{normest} to \eqref{normeq} and $\tilde{u}$ we obtain \eqref{boundres1}.
\end{proof}

\section{Proof of the main result} \label{finalproof}

\begin{proof}[Proof of Theorem \ref{mainthm}]

By replacing $u$ by $u/A$ (i.e. setting $\kappa=1/A$, $\tau=1$ in \eqref{normsol}-\eqref{normcoef})  we may assume $A=1$ in what follows.

\smallskip

\noindent

\textit{Proof of the boundary estimate \eqref{thmest}}. We consider the universal parameters $\delta_0$ and $\rho_0$ given previously in Theorem \ref{smallgrad}. By translating the origin and replacing $u$ by $u-u(x_0)$ we again assume $x_0=0$ and $u(0)=0$. Our analysis is going to be around the following parameter:
\begin{equation}\label{around}
\kappa:=\left(\frac{|Du(0)|}{\delta_0}\right)^{\frac{1}{\alpha}},
\end{equation}
for $\alpha$ given in \eqref{boundexp}.
By using the normalization
$$
\tilde u = \frac{\delta_0\rho_0^\alpha}{|Du(0)|}u\qquad (\mbox{if }\; |Du(0)|\not=0),
$$
we easily see that we can suppose $$\kappa \leq \rho_0.$$

We observe that  \eqref{thmest} is obvious if $|x|=|x-x_0|\ge \rho_0$, since $\rho_0$ is universal and \eqref{C1omega} holds. We now split our analysis into two cases.

\medskip

\noindent
$\bullet$ Case 1. Assume $\rho>0$ is such that $\kappa\le \rho\le\rho_0$. Then we have the following control
$$
|Du(0)| \leq \delta_0\cdot \rho^{\alpha}.
$$
This allows us to apply Theorem \ref{smallgrad}, obtaining
$$
\sup\limits_{\Omega_\rho^+} |u(x)| \leq C\,\rho^{1+\alpha}.
$$
By the last two estimates we conclude that  \eqref{thmest} follows, for each $x$ with $|x|=|x-x_0|=\rho$.

\medskip

\noindent
$\bullet$ Case 2. Assume $0 < \rho < \kappa \leq \rho_0$. We define the rescaled function
\begin{equation}\label{deftheta}
\vartheta(x):= \frac{u(\kappa x)}{\kappa^{1+\alpha}} \quad \mbox{in} \quad  \Omega_{1,\kappa}^+.
\end{equation}
where $\Omega_{1,\kappa}^+$ is defined as in \eqref{rescdom} or \eqref{normcoef}. As in the previous section we see  that $\vartheta$ solves a modified equation in $\Omega_{1,\kappa}^+$ (namely, \eqref{equal1} with $r_j$ substituted by $\kappa$), with the same universal dependence on the parameters, a bounded right-hand side and a bounded $\alpha'$-seminorm of $\vartheta$ on the boundary (as in \eqref{contr}).

On the other hand, from \eqref{around} we have
$$
|Du(0)| = \delta_0 \,\kappa^\alpha.
$$
so, by applying Theorem \ref{smallgrad} again (precisely for the radius $\kappa$), we get
$$
\|\vartheta\|_{L^\infty(\Omega_{1,\kappa}^+)} = \sup\limits_{x \in \Omega_{\kappa}^+}\frac{|u(x)|}{\kappa^{1+\alpha}} \leq C.
$$
By applying \eqref{C1omega} to the equation satisfied by $\vartheta$, we get
\begin{equation}\label{thet}
|D\vartheta(x)-D\vartheta(0)|\le C \,\bar\omega(|x|)\quad\mbox{if }\;x\in \Omega_{1/2,\kappa}^+.
\end{equation}
 But since
$$
|D\vartheta(0)|=\delta_0,
$$
the estimate \eqref{thet} provides a universal small radius $\mu>0$ such that
$$
\delta_0/2 \leq |D\vartheta(x)| \leq 2\delta_0 \quad \mbox{for each} \quad x\in\Omega_{\mu,\kappa}^+.
$$
Thus,  $\vartheta$ satisfies the uniformly elliptic equation with bounded right-hand side
$$
\tilde F(D^2\vartheta,x) = (\tilde H(D\vartheta,x))^{-1} \tilde f(x)
$$
in the region $\Omega_{\mu,\kappa}^+$. So we can apply Theorem \ref{unifthm} (properly rescaled) and obtain
$$
\sup\limits_{\Omega_{\tau,\kappa}^+}|\vartheta(x)-D\vartheta(0)\cdot x| \leq C\,\tau^{1+\alpha\,'},
$$
for any $0<\tau\leq \mu/2$. Hence
\begin{equation}\label{loc1}
\sup\limits_{\Omega_{\rho}^+}|u(x)-Du(0)\cdot x| \leq C\,\rho^{1+\alpha}
\end{equation}
for $0<\rho \leq \kappa\mu/2$, since $\alpha \leq \alpha'$.

We now extend \eqref{loc1} to radii $\rho \in (\mu\kappa/2,\kappa)$. If $\rho$ is one such radius, by using Case 1 with $\rho=\kappa$ we get
$$
\begin{array}{rcl}
\sup\limits_{\Omega_{\rho}^+}|u(x)-Du(0)\cdot x| & \leq & \sup\limits_{\Omega_{\kappa}^+}|u(x)-Du(0)\cdot x| \\[0.5cm]
 & \leq & C\,\kappa^{1+\alpha} \\
 & \leq & C\, \left( {2}/{\mu} \right)^{1+\alpha} \rho^{1+\alpha},
\end{array}
$$
which concludes the proof of  \eqref{thmest}.

\medskip

\noindent
\textit{Proof of the global estimate \eqref{thmest2}}. Again, as a simple adaptation of the proofs in the appendices of \cite{BFM}-\cite{BGMW} shows, to obtain \eqref{thmest2} it is sufficient to prove that
$$
|u(y)-u(x)-Du(x)\cdot (y-x)| \leq C  \,|y-x|^{1+\beta}, \qquad x,y\in {\Omega_{1/2}^+}.
$$
Let $d_1<1$ be a number such that the distance function to the boundary of $\Omega$ is smooth in the set dist$(x,\partial\Omega) <d_1$. Since  we already proved the boundary $C^{1,\alpha}$ estimates, and by \cite{ART1} have the interior $C^{1,\sigma}$ regularity, it is enough to suppose that one of $x,y$, say $x$, is such that $0<\mathrm{dist}(x,\partial\Omega) <d_1/2$.
 Fix one such $x$ and denote $\bar x $ the point at $\partial \Omega$ for which $d_0=|x-\bar x| = \mathrm{dist}(x,\partial \Omega)$. We can assume again that $\bar x=0$ and $u(0)=0$, by translating the origin and removing a constant from $u$.

 We denote with $L_{x}^u (y)$ the supporting hyperplane of the function $u$ at the point $x$, that is $L_{x}^u (y) = u(x) + Du(x) \cdot (y-x)$, and set
 $$v(y):= u(y) - L_{\bar{x}}^u(y) = u(y)  - Du(0) \cdot y.$$

From the boundary estimate which we just proved we know that
\begin{align} \label{e1}
|v(y)|&\leq C  |y|^{1+\alpha}, \quad \mbox{ for }\; y\in \Omega_{3/4}^+.
\end{align}

We will again use the  parameter $\kappa$ from \eqref{around}
and divide the proof into cases, this time according to how $\kappa$ and $d_0$ compare.

Exactly as in Case 2 above, by using the rescaling \eqref{deftheta} we obtain a universal $\mu>0$ such that by the global estimate in Theorem \ref{unifthm} and $\beta\le\bar\beta$
\begin{align}
|u(y)-u(x)-Du(x)\cdot (y-x)| &\leq C  \,|y-x|^{1+\beta}\label{glob1}\\
|Du(y)-Du(x)|&\le  C  \,|y-x|^{\beta},\label{glob2}
\end{align}
provided $|x|,|y|\le 3\mu\kappa/4$. In particular, if $|x|\le 3\mu\kappa/4$ we have
$$
|Dv(x)|=|Du(x)-Du(0)|\le C\,d_0^\beta.
$$
So if $d_0=|x|\le \mu\kappa/2$ and $y\in B_{d_0/2}(x)$, the desired inequality \eqref{thmest2} is given by \eqref{glob1}, while if $|x|\le \mu\kappa/2$ and $y\not \in B_{d_0/2}(x)$ we can write
\begin{align}
|u(y)-u(x)-Du(x)\cdot (y-x)|&=|u(y) - (L_{\bar{x}}^u(y) + L_{x}^v(y) )|  \leq |u(y) - L_{\bar{x}}^u(y)| + |L_x^v(y)| \nonumber\\
&\leq C\,|y|^{1+\alpha} + |v(x)| + |Dv(x)|\,|y-x|\label{loci1}
\end{align}
and since
\begin{equation}\label{loci2}
\begin{cases}
|y| \leq |y-x| + |x| = |y-x| + d_0 \leq 3|y-x| \\
|v(x)| \leq C \,|x|^{1+\alpha} = C \, d_0^{1+\alpha} \leq 2^{1+\alpha}C \, |y-x|^{1+\alpha} \\
|Dv(x)|\,|y-x| \leq C \, d_0^{\beta}\, |y-x| \leq 2^{1+\beta} C \,|y-x|^{1+\beta}
\end{cases}
\end{equation}
we obtain \eqref{thmest2} again.

Thus from now on we can assume that $d_0\ge \mu\kappa/2$, that is,
\begin{equation}\label{smallgr}
|Du(0)|\le Cd_0^\alpha.
\end{equation}
Set   $$\tilde v (\tilde y) =v(x+d_0\tilde y)=v(y),\quad\mbox{ and } \quad \tilde w(\tilde y) = \tilde v(\tilde y) + d_0|Du(0)|\tilde y.$$ Since $u(y) = v(y) + Du(0)\cdot y$, it is easy to check that $\tilde w$ is a solution of an equation
\begin{align} \label{e2}
\tilde{H} (D\tilde w(\tilde y),\tilde y)\tilde{F}(D^2\tilde w(\tilde y),\tilde  y) = d_0^{2+\gamma}\tilde f(\tilde y), \quad \tilde f(\tilde y) ={f}(x+d_0\tilde y),
\end{align}
in the unit ball $B_1(0)$, where $
\tilde{H}, \tilde{F}$ satisfy \eqref{HH}.
By applying the interior regularity result from \cite{ART1} we obtain that
\begin{equation}\label{loc3}
\tilde w(\tilde y) - \tilde w (0) - D\tilde w(\tilde  y)\cdot \tilde  y \le C\left(\|\tilde w\|_{L^\infty(B_1)} + d_0^{\frac{2+\gamma}{1+\gamma}}\|\tilde f\|_{L^\infty(B_1)}^{\frac{1}{1+\gamma}}\right)|\tilde y|^{1+\beta},\quad\mbox{ for }\; \tilde y\in B_{1/2},
\end{equation}
and
\begin{equation}\label{loc4}
D\tilde w(0)\le C\left(\|\tilde w\|_{L^\infty(B_1)} + d_0^{\frac{2+\gamma}{1+\gamma}}\|\tilde f\|_{L^\infty(B_1)}^{\frac{1}{1+\gamma}}\right).
\end{equation}
By \eqref{e1} and \eqref{smallgr} we have
$$
\|\tilde w\|_{L^\infty(B_1)}\le C d_0^{1+\alpha}
$$
so \eqref{loc3} implies
\begin{align*}
|u(y)-u(x)-Du(x)\cdot (y-x)|& = |v(y)-v(x)-Dv(x)\cdot (y-x)|\\
& = |\tilde w(\tilde y) - \tilde w (0) - D\tilde w(\tilde  y)|\\
&\le C\left(d_0^{1+\alpha} + d_0^{1+\frac{1}{1+\gamma}}\|\tilde f\|_{L^\infty(B_1)}^{\frac{1}{1+\gamma}}\right)
\frac{|y-x|^{1+\beta}}{d_0^{1+\beta}}\\
&\le C|y-x|^{1+\beta},
\end{align*}
provided $y\in B_{d_0/2}(x)$. On the other hand \eqref{smallgr} and \eqref{loc4} imply
$$
|Dv(x)|= \frac{1}{d_0}|D\tilde v(0)| =  \frac{1}{d_0}|D\tilde w(0)-d_0Du(0)|\le C\,d_0^\beta,
$$ so if $y\not\in B_{d_0/2}(x)$ we can repeat \eqref{loci1}-\eqref{loci2}, and conclude the proof.

\end{proof}
\bigskip

{\small \noindent{\bf Acknowledgments.} D.J.A. is partially supported by Conselho Nacional de Desenvolvimento Cient\'ifico e Tecnol\'ogico and Para\'iba State Research Foundation (FAPESQ) grant 2019/0014. D.J.A. thanks the Abdus Salam International Centre for Theoretical Physics (ICTP) for their hospitality during
his research visits. B.S. is partially supported by Funda\c c\~ao Carlos Chagas Filho de Amparo \`a Pesquisa do Estado do Rio de Janeiro (FAPERJ) grant E-26/203-015/2017; and Conselho Nacional de Desenvolvimento Cient\'ifico e Tecnol\'ogico (CNPq) grants 427056/2018-7, 310989/2018-3.}


\vspace{1cm}
\noindent \textsc{Dami\~ao J. Ara\'ujo} \hfill \textsc{Boyan Sirakov}\\
Universidade Federal da Para\'iba\hfill  Pontif\'icia Universidade Cat\'olica - PUC\\
Department of Mathematics \hfill Department of Mathematics \\
Jo\~ao Pessoa, PB 58059-900 \hfill Rio de Janeiro, RJ 22451-900\\
Brazil \hfill Brazil \\
\texttt{araujo@mat.ufpb.br} \hfill
\texttt{bsirakov@mat.puc-rio.br}

\end{document}